\documentclass[12pt,a4paper]{article}

\usepackage[reqno]{amsmath}
\usepackage{amssymb}
\usepackage{amsthm}
\usepackage{upref}
\usepackage{amscd}
\usepackage[T1]{fontenc}
\usepackage{times}
\usepackage{xcolor}
\usepackage{array}
\usepackage{multirow}
\usepackage{makecell}
\usepackage{mathrsfs}
\usepackage{graphicx}

\theoremstyle{plain}
\newtheorem{theorem}{Theorem}
\newtheorem{corollary}[theorem]{Corollary}
\newtheorem{lemma}[theorem]{Lemma}
\newtheorem{proposition}[theorem]{Proposition}
\newtheorem{definition}[theorem]{Definition}

\newtheorem{remark}[theorem]{Remark}

\numberwithin{equation}{section}
\numberwithin{theorem}{section}

\begin{document}

\title{Lower and Upper Bounds for Sums of Eigenvalues of the
	Fractional-Logarithmic Laplacian}

\author{
	Maha Daoud$^{\mathrm{a},*}$ \qquad
	Hichem Hajaiej$^{\mathrm{b}}$
	\\[0.7em]
	{\small $^{\mathrm{a}}$Département Génie Mathématique et Modélisation, INSA Toulouse,}\\
	{\small 31400 Toulouse, France}\\
	{\small Email: \texttt{mahaadaoud@gmail.com}}\\[0.5em]
	{\small $^{\mathrm{b}}$Department of Mathematics, California State University at Los Angeles,}\\
	{\small Los Angeles, CA 90032, USA}\\
	{\small Email: \texttt{hhajaie@calstatela.edu}}\\[0.5em]
	{\small $^{*}$Corresponding author}
}

\date{}
\maketitle

\setcounter{secnumdepth}{2}
\setcounter{tocdepth}{1}

\tableofcontents

\begin{abstract}
In this work, we establish lower and upper bounds for sums of Dirichlet eigenvalues
	of the fractional-logarithmic Laplacian. A main challenge in such a study comes from the fact that this operator has a Fourier symbol that is not globally monotone in its radial variable due to its low-frequency behavior. The bounds have the correct leading-order asymptotics.  To complement the theoretical analysis, we also present a
	one-dimensional finite element approximation of the
	fractional-logarithmic eigenvalue problem. We compare the computed eigenvalue
	sums  with the explicit lower bound and the principal
	asymptotic term.
\end{abstract}


\section{ Introduction and Main Results}

Let $\Omega\subset\mathbb R^n$ be a bounded domain, and let
$$
\lambda_1(\Omega)\leq\lambda_2(\Omega)\leq\cdots\uparrow+\infty
$$
denote the Dirichlet eigenvalues of the operator under consideration. A central topic in spectral theory is to understand the asymptotic growth and bounds for the partial sums

$$
\sum_{j=1}^{k} \lambda_{j}(\Omega), \quad k \in \mathbb{N}.
$$

These sums encode both the high-frequency distribution of the spectrum and refined geometric information on the domain.

For the classical Dirichlet Laplacian $-\Delta$, we write its eigenvalues as

$$
0<\lambda_{1}^{1}(\Omega) \leq \lambda_{2}^{1}(\Omega) \leq \cdots \uparrow+\infty.
$$

Weyl's law asserts that

$$
\lambda_{k}^{1}(\Omega) \sim C_{n}\left(\frac{k}{|\Omega|}\right)^{2 / n} \quad \text { as } k \rightarrow \infty,
$$

where

$$
C_{n}:=4 \pi^{2} \omega_{n}^{-2 / n},
$$

and $\omega_{n}$ denotes the volume of the unit ball in $\mathbb{R}^{n}$, given by $\omega_{n}=\frac{\pi^{n / 2}}{\Gamma\left(\frac{n}{2}+1\right)}$. Therefore,

$$
\sum_{j=1}^{k} \lambda_{j}^{1}(\Omega) \sim \frac{n}{n+2} C_{n}|\Omega|^{-2 / n} k^{1+2 / n}.
$$

At the level of lower bounds, Berezin's semiclassical argument \cite{Ber72} and the celebrated Li-Yau inequality \cite{LY83} yield the sharp leading-order estimate

$$
\sum_{j=1}^{k} \lambda_{j}^{1}(\Omega) \geq \frac{n}{n+2} C_{n}|\Omega|^{-2 / n} k^{1+2 / n},
$$

which is optimal in view of Weyl asymptotics.
This inequality was later refined by Melas, while universal upper bounds for Dirichlet eigenvalues were established by Cheng and Yang; see \cite{CY07, Mel03}. On the other hand, Kröger obtained an asymptotically sharp upper bound for the sums of the first $k$ Dirichlet eigenvalues of the Laplacian, see \cite{Kro94}.

Analogous questions for nonlocal operators have also been extensively studied. For a general overview of the fractional Laplacian and its different realizations,
we refer to \cite{DaoudLaamri} and references therein. For the Dirichlet fractional Laplacian $(-\Delta)^{s}, 0<s<1$, let

$$
0<\lambda_{1}^{s}(\Omega) \leq \lambda_{2}^{s}(\Omega) \leq \cdots \uparrow+\infty
$$

be its eigenvalues. The corresponding Weyl scale is

$$
\lambda_{k}^{s}(\Omega) \sim C_{n, s}\left(\frac{k}{|\Omega|}\right)^{2 s / n}, \quad \sum_{j=1}^{k} \lambda_{j}^{s}(\Omega) \sim \frac{n}{n+2 s} C_{n, s}|\Omega|^{-2 s / n} k^{1+2 s / n},
$$

where

$$
C_{n, s}:=(2 \pi)^{2 s} \omega_{n}^{-2 s / n}.
$$

In particular, the Berezin-Li-Yau type lower bound takes the form

$$
\sum_{j=1}^{k} \lambda_{j}^{s}(\Omega) \geq \frac{n}{n+2 s} C_{n, s}|\Omega|^{-2 s / n} k^{1+2 s / n}.
$$

Such lower bounds were established by Yolcu and Yolcu \cite{YY13}, and later sharpened in several directions.  On the other hand, Kröger-type upper bounds for the sums of the
first $k$ Dirichlet eigenvalues of the fractional Laplacian were
obtained by Wang, Chen and the second author, with the same Weyl
order $k^{1+2s/n}$; see \cite{WCH21}. We also refer to Frank for the corresponding Weyl asymptotics and a broader overview of eigenvalue bounds for fractional and related nonlocal operators; see \cite{Fra16, FG11}. For mixed fractional Laplacians, see also \cite{CBH22}.

A further singular limit of the fractional operator leads to the logarithmic Laplacian. We refer to \cite{CW19} for the basic theory of the Dirichlet problem for the logarithmic Laplacian and to \cite{CV23} for eigenvalue estimates. If

$$
\lambda_{1}^{\ln }(\Omega) \leq \lambda_{2}^{\ln }(\Omega) \leq \cdots \uparrow+\infty
$$

denote the eigenvalues of the Dirichlet logarithmic Laplacian $(-\Delta)^{\ln }$, whose Fourier symbol is $2 \ln |\xi|$, then the asymptotic behavior is fundamentally different from both the local and fractional cases: the growth is no longer governed by a pure power law, but instead by a logarithmic law,

$$
\lambda_{k}^{\ln }(\Omega) \sim \frac{2}{n} \ln k, \quad \sum_{j=1}^{k} \lambda_{j}^{\ln }(\Omega) \sim \frac{2}{n} k \ln k.
$$

In \cite{CV23}, Chen and Véron obtained lower and upper bounds for the sums of the first $k$ eigenvalues by adapting Li-Yau type and Kröger type arguments, and proved in particular that the leading asymptotic term of the eigenvalue sums is independent of the volume of the domain. This logarithmic regime provides a natural bridge between homogeneous nonlocal symbols and their logarithmic corrections.

The aim of the present paper is to study the corresponding spectral problem for the fractional-logarithmic Laplacian $(-\Delta)^{s+\ln }$, whose Fourier symbol is

$$
|\xi|^{2 s} \ln |\xi|^{2}.
$$

This operator was introduced very recently by R. Chen together with H. Chen and D. Hauer in \cite{CCH26}, where the fractional-logarithmic Laplacian was defined, for the first time, as the derivative of the fractional Laplacian with respect to the order parameter, and a systematic study of its basic properties was initiated. In particular, the authors established several equivalent realizations of $(-\Delta)^{s+\ln }$, including
the singular-integral formulation, the Fourier-multiplier representation, the spectral-calculus definition, and an extension characterization. They also developed the natural energy framework both on $\mathbb{R}^{n}$ and on bounded domains, proved embedding results and compactness properties, analyzed the Poisson problem, and studied the associated Dirichlet eigenvalue problem, including qualitative spectral properties and a Weyl-type asymptotic law for the counting function and for the $k$-th eigenvalue. More recently, in \cite{Che26}, R. Chen further developed the potential theory of the fractional-logarithmic Laplacian $(-\Delta)^{s+\ln }$ and of its inhomogeneous counterpart $(\lambda I-\Delta)^{s+\ln }$ with $\lambda>1$. In addition, that work established global regularity theory for solutions and the critical compact embedding theory for logarithmic Bessel spaces.

In [6], the authors established the existence of the eigenvalues
of the Dirichlet problem for $(-\Delta)^{s+\ln}$, denoted by
$$
\lambda_1^{s+\ln}(\Omega)
\leq\lambda_2^{s+\ln}(\Omega)
\leq\cdots\uparrow+\infty;
$$
see [6, Theorem 1.5]. Moreover, they proved that

$$
\lambda_{k}^{s+\ln }(\Omega) \sim \frac{2}{n}(2 \pi)^{2 s}\left(\omega_{n}|\Omega|\right)^{-2 s / n} k^{2 s / n} \ln k \quad \text { as } k \rightarrow \infty .
$$

The Weyl asymptotics above exhibit the mixed
	fractional-logarithmic nature of the operator: the power growth
	$k^{2s/n}$ inherited from the fractional Laplacian is accompanied
	by the logarithmic correction $\ln k$. Consequently, by summation,
	one is naturally led to the asymptotic formula
	(see Proposition~\ref{Proposition 2.3})
	$$
	\sum_{j=1}^{k}\lambda_j^{s+\ln}(\Omega)
	\sim
	\frac{2}{n+2s}(2\pi)^{2s}
	(\omega_n|\Omega|)^{-2s/n}
	k^{1+2s/n}\ln k
	\qquad\text{as }k\to\infty.
	$$

	The main purpose of the present paper is to obtain explicit lower
	and upper bounds for these eigenvalue sums with this sharp
	leading-order behavior. The lower estimate is obtained through a
	Berezin--Li--Yau type argument adapted to the non-monotone symbol
	$|\xi|^{2s}\ln|\xi|^2$, while the upper estimate is based on a
	Kröger-type construction combined with a boundary cutoff estimate.
	In particular, the upper-bound argument requires a careful control
	of the error produced by the cutoff near $\partial\Omega$.

%
%

To complement the theoretical analysis, we also include a
	one-dimensional numerical illustration on $\Omega=(0,1)$. We
	discretize the integral fractional Laplacian using conforming $P_1$
	finite elements and approximate the fractional-logarithmic stiffness
	matrix by a centered finite difference with respect to the order
	parameter. We compare the resulting discrete eigenvalue sums with the
	explicit lower bound and the principal asymptotic term, and we examine
	their sensitivity with respect to the finite-difference step and the
	mesh size. The numerical experiments are consistent with the predicted
	spectral growth and the correction order appearing in the theoretical
	estimates.

\section{ Lower Bound Estimate}

The ideas developed in this section are inspired by the classical Berezin-Li-Yau method and by its nonlocal variants. A key difficulty here is that the symbol $|\xi|^{2 s} \ln |\xi|^{2}$ is not globally monotone in the radial variable, due to its low-frequency behavior. Nevertheless, its high-frequency growth is sufficiently rigid to permit sharp control of the leading term. In this sense, the fractional-logarithmic case is intermediate between\\
the pure fractional case, where the symbol is homogeneous and monotone, and the pure logarithmic case, where the growth is much weaker and the optimization becomes more delicate.

\begin{definition}\label{Definition 2.1}
   Let $L$ be a positive measurable function defined on some neighborhood $[X, \infty)$ of infinity, where $X>0$. We say that $L$ is slowly varying at infinity (in the sense of Karamata) if for every $t>0$,

$$
\lim _{x \rightarrow \infty} \frac{L(t x)}{L(x)}=1.
$$
\end{definition}

This is the classical definition; see \cite[ Definition 1.2.1]{BGT89}. Typical examples of slowly varying functions at infinity are

$$
L(x)=(\ln x)^{\beta}, \quad L(x)=(\ln \ln x)^{\gamma}, \quad \beta, \gamma \in \mathbb{R},
$$

defined for all sufficiently large $x$.\\[0pt]
The following result is a standard Tauberian theorem of Karamata for power series; see \cite[Corollary 1.7.3]{BGT89}.

\begin{proposition}\label{Proposition 2.2}
 Let $\rho>0$, let $L$ be slowly varying at infinity, and let $\left\{a_{k}\right\}_{k \geq 0}$ be a nonnegative sequence which is ultimately monotone. Assume that

$$
a_{k} \sim \frac{c}{\Gamma(\rho)} k^{\rho-1} L(k) \quad \text { as } k \rightarrow \infty
$$

for some constant $c>0$. Then

$$
\sum_{j=0}^{k} a_{j} \sim \frac{c}{\Gamma(1+\rho)} k^{\rho} L(k) \quad \text { as } k \rightarrow \infty.
$$
\end{proposition}

\begin{proposition}\label{Proposition 2.3}
   Assume that

$$
\lambda_{k}^{s+\ln }(\Omega) \sim A k^{2 s / n} \ln k \quad \text { as } k \rightarrow \infty
$$

where

$$
A=\frac{2}{n}(2 \pi)^{2 s}\left(\omega_{n}|\Omega|\right)^{-2 s / n}.
$$

Then

$$
\sum_{j=1}^{k} \lambda_{j}^{s+\ln }(\Omega) \sim \frac{2}{n+2 s}(2 \pi)^{2 s}\left(\omega_{n}|\Omega|\right)^{-2 s / n} k^{1+2 s / n} \ln k \quad \text { as } k \rightarrow \infty.
$$
\end{proposition}

\begin{proof}
   Set

$$
a_{k}:=\lambda_{k}^{s+\ln }(\Omega), \quad L(k):=\ln k, \quad \rho:=1+\frac{2 s}{n}.
$$

Since

$$
a_{k} \sim A k^{2 s / n} \ln k \quad \text { with } A>0,
$$

there exists $k_{0} \in \mathbb{N}$ such that $a_{k}>0$ for all $k \geq k_{0}$. Moreover, since $\left\{\lambda_{k}^{s+\ln }(\Omega)\right\}_{k \geq 1}$ is nondecreasing, the tail sequence $\left\{a_{k}\right\}_{k \geq k_{0}}$ is nonnegative and ultimately monotone. Hence Proposition \ref{Proposition 2.2} applies to the shifted sequence

$$
b_{m}:=a_{m+k_{0}}=\lambda_{m+k_{0}}^{s+\ln }(\Omega), \quad m \geq 0.
$$

Since

$$
b_{m} \sim A m^{2 s / n} \ln m=A m^{\rho-1} L(m),
$$

we may write

$$
b_{m} \sim \frac{c}{\Gamma(\rho)} m^{\rho-1} L(m), \quad c:=A \Gamma(\rho).
$$

Therefore,

$$
\sum_{m=0}^{k} b_{m} \sim \frac{c}{\Gamma(1+\rho)} k^{\rho} L(k).
$$

Using $\Gamma(1+\rho)=\rho \Gamma(\rho)$, we obtain

$$
\frac{c}{\Gamma(1+\rho)}=\frac{A}{\rho}=\frac{A}{1+\frac{2 s}{n}}=\frac{n A}{n+2 s} .
$$

Thus

$$
\sum_{m=0}^{k} b_{m} \sim \frac{n A}{n+2 s} k^{1+2 s / n} \ln k.
$$

Finally,

$$
\sum_{j=1}^{k} \lambda_{j}^{s+\ln }(\Omega)=\sum_{j=1}^{k_{0}-1} \lambda_{j}^{s+\ln }(\Omega)+\sum_{m=0}^{k-k_{0}} \lambda_{m+k_{0}}^{s+\ln }(\Omega).
$$

The first term is constant, hence negligible, and replacing $k-k_{0}$ by $k$ does not affect the asymptotics. Therefore,

$$
\sum_{j=1}^{k} \lambda_{j}^{s+\ln }(\Omega) \sim \frac{n A}{n+2 s} k^{1+2 s / n} \ln k.
$$

Substituting the value of $A$, we conclude that

$$
\sum_{j=1}^{k} \lambda_{j}^{s+\ln }(\Omega) \sim \frac{2}{n+2 s}(2 \pi)^{2 s}\left(\omega_{n}|\Omega|\right)^{-2 s / n} k^{1+2 s / n} \ln k,
$$

as claimed.
\end{proof}

\begin{lemma}\label{Lemma 2.4}
   Let $s>0, M_{1}>0$, and let $f: \mathbb{R}^{n} \rightarrow\left[0, M_{1}\right]$ be measurable. Define

$$
A:=\int_{\mathbb{R}^{n}} f(z) d z \quad \text { and } \quad M_{2}:=\int_{\mathbb{R}^{n}}|z|^{2 s} \ln |z|^{2} f(z) d z.
$$

Then

$$
M_{2} \geq-\frac{2 n \omega_{n}}{(n+2 s)^{2}} M_{1}.
$$

Moreover, assume that

$$
A \geq M_{1} \omega_{n}
$$

where $\omega_{n}=\left|B_{1}\right|$ denotes the volume of the unit ball in $\mathbb{R}^{n}$. Then

$$
M_{2} \geq \frac{2}{n+2 s}\left(M_{1} \omega_{n}\right)^{-2 s / n} A^{1+2 s / n}\left(\ln \frac{A}{M_{1} \omega_{n}}-\frac{n}{n+2 s}\right).
$$

In particular, as $A \rightarrow \infty$,

$$
M_{2} \geq \frac{2}{n+2 s}\left(M_{1} \omega_{n}\right)^{-2 s / n} A^{1+2 s / n} \ln A+O\left(A^{1+2 s / n}\right).
$$
\end{lemma}

\begin{proof}
   Since $|z|^{2 s} \ln |z|^{2} \geq 0$ for $|z| \geq 1$, we have

$$
M_{2}=\int_{B_{1}}|z|^{2 s} \ln |z|^{2} f(z) d z+\int_{B_{1}^{c}}|z|^{2 s} \ln |z|^{2} f(z) d z \geq \int_{B_{1}}|z|^{2 s} \ln |z|^{2} f(z) d z.
$$

Now, on $B_{1}$ one has $|z|^{2 s} \ln |z|^{2} \leq 0$, and since $0 \leq f \leq M_{1}$, it follows that

$$
\int_{B_{1}}|z|^{2 s} \ln |z|^{2} f(z) d z \geq M_{1} \int_{B_{1}}|z|^{2 s} \ln |z|^{2} d z.
$$

Using polar coordinates,

$$
\int_{B_{1}}|z|^{2 s} \ln |z|^{2} d z=n \omega_{n} \int_{0}^{1} r^{n+2 s-1} \ln r^{2} d r=-\frac{2 n \omega_{n}}{(n+2 s)^{2}}
$$

Therefore,

$$
M_{2} \geq-\frac{2 n \omega_{n}}{(n+2 s)^{2}} M_{1}.
$$

Set

$$
w(r):=r^{2 s} \ln r^{2}, \quad r>0.
$$

We first note that $w$ is strictly increasing on $[1, \infty)$. Indeed,

$$
w^{\prime}(r)=2 r^{2 s-1}(2 s \ln r+1)>0, \quad r \geq 1.
$$

Moreover, for $0<r<1$ we have $w(r)<0$, and hence for every $R \geq 1$,

$$
w(r) \leq w(R), \quad 0<r<R.
$$

Therefore, for every $R \geq 1$ and every $z \in \mathbb{R}^{n}$,

$$
(w(|z|)-w(R))\left(f(z)-M_{1} \mathbf{1}_{B_{R}}(z)\right) \geq 0.
$$

Indeed, if $|z|<R$, then $w(|z|) \leq w(R)$ and $f(z)-M_{1} \leq 0$; if $|z| \geq R$, then $w(|z|) \geq w(R)$ and $f(z) \geq 0$.\\
Integrating the above inequality over $\mathbb{R}^{n}$, we obtain

$$
\int_{\mathbb{R}^{n}}(w(|z|)-w(R))\left(f(z)-M_{1} \mathbf{1}_{B_{R}}(z)\right) d z \geq 0.
$$

Expanding this inequality gives

$$
M_{2}-w(R) A-M_{1} \int_{B_{R}} w(|z|) d z+M_{1} w(R)\left|B_{R}\right| \geq 0.
$$

Hence, for every $R \geq 1$,

$$
M_{2} \geq \Psi_{A}(R),
$$

where

$$
\Psi_{A}(R):=A w(R)-M_{1}\left(w(R)\left|B_{R}\right|-\int_{B_{R}} w(|z|) d z\right).
$$

We now optimize in $R$. Since $\left|B_{R}\right|=\omega_{n} R^{n}$, we have

$$
\frac{d}{d R} \int_{B_{R}} w(|z|) d z=n \omega_{n} R^{n-1} w(R),
$$

and

$$
\frac{d}{d R}\left(w(R)\left|B_{R}\right|\right)=w^{\prime}(R) \omega_{n} R^{n}+n \omega_{n} R^{n-1} w(R).
$$

Therefore,

$$
\Psi_{A}^{\prime}(R)=w^{\prime}(R)\left(A-M_{1} \omega_{n} R^{n}\right).
$$

Since $w^{\prime}(R)>0$ for $R \geq 1$, it follows that $\Psi_{A}$ attains its maximum at

$$
R_{A}:=\left(\frac{A}{M_{1} \omega_{n}}\right)^{1 / n},
$$

provided $R_{A} \geq 1$, namely provided $A \geq M_{1} \omega_{n}$, which is our assumption. Thus,

$$
M_{2} \geq \Psi_{A}\left(R_{A}\right).
$$

Because $M_{1}\left|B_{R_{A}}\right|=M_{1} \omega_{n} R_{A}^{n}=A$, we have

$$
\Psi_{A}\left(R_{A}\right)=M_{1} \int_{B_{R_{A}}} w(|z|) d z.
$$

It remains to compute the integral. By polar coordinates,

$$
\int_{B_{R}}|z|^{2 s} \ln |z|^{2} d z=n \omega_{n} \int_{0}^{R} r^{n+2 s-1} \ln r^{2} d r.
$$

A direct computation yields

$$
\int_{0}^{R} r^{n+2 s-1} \ln r^{2} d r=R^{n+2 s}\left(\frac{\ln R^{2}}{n+2 s}-\frac{2}{(n+2 s)^{2}}\right).
$$

Hence

$$
\int_{B_{R}}|z|^{2 s} \ln |z|^{2} d z=n \omega_{n} R^{n+2 s}\left(\frac{\ln R^{2}}{n+2 s}-\frac{2}{(n+2 s)^{2}}\right) .
$$

Substituting $R=R_{A}$, we obtain

$$
M_{2} \geq n \omega_{n} M_{1} R_{A}^{n+2 s}\left(\frac{\ln R_{A}^{2}}{n+2 s}-\frac{2}{(n+2 s)^{2}}\right).
$$

Since

$$
R_{A}^{n}=\frac{A}{M_{1} \omega_{n}}, \quad R_{A}^{n+2 s}=\left(\frac{A}{M_{1} \omega_{n}}\right)^{1+2 s / n}, \quad \ln R_{A}^{2}=\frac{2}{n} \ln \frac{A}{M_{1} \omega_{n}},
$$

it follows that

$$
M_{2} \geq \frac{2}{n+2 s}\left(M_{1} \omega_{n}\right)^{-2 s / n} A^{1+2 s / n}\left(\ln \frac{A}{M_{1} \omega_{n}}-\frac{n}{n+2 s}\right).
$$

This proves the claimed lower bound. Finally, expanding the term inside the parentheses

$$
M_{2} \geq \frac{2}{n+2 s}\left(M_{1} \omega_{n}\right)^{-2 s / n} A^{1+2 s / n} \ln A+O\left(A^{1+2 s / n}\right) \quad \text { as } A \rightarrow \infty,
$$

which completes the proof.
\end{proof}

\begin{theorem}\label{Theorem 2.5}
  Let $\Omega \subset \mathbb{R}^{n}$ be a bounded domain, and let

$$
\lambda_{1}^{s+\ln }(\Omega) \leq \lambda_{2}^{s+\ln }(\Omega) \leq \cdots \uparrow+\infty
$$

be the Dirichlet eigenvalues of $(-\Delta)^{s+\ln }$. Then, for every integer $k \geq 1$, we have

$$
\sum_{j=1}^{k} \lambda_{j}^{s+\ln }(\Omega) \geq-(2 \pi)^{-n} \frac{2 n \omega_{n}}{(n+2 s)^{2}}|\Omega|,
$$

and if $k \geq(2 \pi)^{-n} \omega_{n}|\Omega|$, then

$$
\sum_{j=1}^{k} \lambda_{j}^{s+\ln }(\Omega) \geq \frac{2}{n+2 s}(2 \pi)^{2 s}\left(\omega_{n}|\Omega|\right)^{-2 s / n} k^{1+2 s / n}\left(\ln \frac{(2 \pi)^{n} k}{\omega_{n}|\Omega|}-\frac{n}{n+2 s}\right).
$$
\end{theorem}

\begin{proof} Let $\left\{\phi_{j}\right\}_{j \geq 1}$ be an orthonormal family of eigenfunctions in $L^{2}(\Omega)$ associated with $\lambda_{j}^{s+\ln }(\Omega)$, extended by zero outside $\Omega$. We adopt the Fourier transform

$$
\widehat{u}(\xi):=(2 \pi)^{-n / 2} \int_{\mathbb{R}^{n}} u(x) e^{-i x \cdot \xi} d x, \quad \xi \in \mathbb{R}^{n}.
$$

Define

$$
f(\xi):=\sum_{j=1}^{k}\left|\widehat{\phi_{j}}(\xi)\right|^{2}, \quad \xi \in \mathbb{R}^{n}.
$$

Clearly, $f(\xi) \geq 0$ for every $\xi \in \mathbb{R}^{n}$. For each fixed $\xi \in \mathbb{R}^{n}$, we have

$$
\widehat{\phi_{j}}(\xi)=(2 \pi)^{-n / 2} \int_{\mathbb{R}^{n}} \phi_{j}(x) e^{-i x \cdot \xi} d x=(2 \pi)^{-n / 2} \int_{\Omega} \phi_{j}(x) e^{-i x \cdot \xi} d x=\left\langle\phi_{j},(2 \pi)^{-n / 2} e^{-i x \cdot \xi}\right\rangle_{L^{2}(\Omega)},
$$

since each $\phi_{j}$ vanishes outside $\Omega$. Therefore,

$$
f(\xi)=\sum_{j=1}^{k}\left|\left\langle\phi_{j},(2 \pi)^{-n / 2} e^{-i x \cdot \xi}\right\rangle_{L^{2}(\Omega)}\right|^{2} .
$$

Because $\left\{\phi_{j}\right\}_{j \geq 1}$ is an orthonormal family in $L^{2}(\Omega)$, Bessel's inequality yields

$$
f(\xi) \leq\left\|(2 \pi)^{-n / 2} e^{-i x \cdot \xi}\right\|_{L^{2}(\Omega)}^{2}.
$$

Since $\left|e^{-i x \cdot \xi}\right|=1$ for every $x, \xi \in \mathbb{R}^{n}$, we obtain

$$
\left\|(2 \pi)^{-n / 2} e^{-i x \cdot \xi}\right\|_{L^{2}(\Omega)}^{2}=(2 \pi)^{-n} \int_{\Omega} 1 d x=(2 \pi)^{-n}|\Omega|.
$$

Hence, for every $\xi \in \mathbb{R}^{n}$,

$$
0 \leq f(\xi) \leq(2 \pi)^{-n}|\Omega| .
$$

Thus we may pick 

$$
M_{1}:=(2 \pi)^{-n}|\Omega|.
$$

Next, by Plancherel's theorem,

$$
\int_{\mathbb{R}^{n}} f(\xi) d \xi=\sum_{j=1}^{k} \int_{\mathbb{R}^{n}}\left|\widehat{\phi_{j}}(\xi)\right|^{2} d \xi=\sum_{j=1}^{k} \int_{\Omega}\left|\phi_{j}(x)\right|^{2} d x=k.
$$

Therefore,

$$
A:=\int_{\mathbb{R}^{n}} f(\xi) d \xi=k.
$$

Moreover, since under the above Fourier convention the operator $(-\Delta)^{s+\ln }$ has symbol $|\xi|^{2 s} \ln |\xi|^{2}$, we have

$$
\lambda_{j}^{s+\ln }(\Omega)=\int_{\mathbb{R}^{n}}|\xi|^{2 s} \ln |\xi|^{2}\left|\widehat{\phi_{j}}(\xi)\right|^{2} d \xi.
$$

Summing over $j=1, \ldots, k$, we obtain

$$
\sum_{j=1}^{k} \lambda_{j}^{s+\ln }(\Omega)=\int_{\mathbb{R}^{n}}|\xi|^{2 s} \ln |\xi|^{2} f(\xi) d \xi.
$$

Thus, with

$$
M_{2}:=\int_{\mathbb{R}^{n}}|\xi|^{2 s} \ln |\xi|^{2} f(\xi) d \xi=\sum_{j=1}^{k} \lambda_{j}^{s+\ln }(\Omega)
$$

all the assumptions of Lemma \ref{Lemma 2.4} are satisfied provided

$$
A=k \geq M_{1} \omega_{n}=(2 \pi)^{-n} \omega_{n}|\Omega| .
$$

Applying Lemma \ref{Lemma 2.4} with

$$
M_{1}=(2 \pi)^{-n}|\Omega| \quad \text { and } \quad A=k,
$$

we obtain

$$
\sum_{j=1}^{k} \lambda_{j}^{s+\ln }(\Omega)=M_{2} \geq \frac{2}{n+2 s}\left((2 \pi)^{-n} \omega_{n}|\Omega|\right)^{-2 s / n} k^{1+2 s / n}\left(\ln \frac{k}{(2 \pi)^{-n} \omega_{n}|\Omega|}-\frac{n}{n+2 s}\right).
$$

Since

$$
\left((2 \pi)^{-n} \omega_{n}|\Omega|\right)^{-2 s / n}=(2 \pi)^{2 s}\left(\omega_{n}|\Omega|\right)^{-2 s / n},
$$

this becomes

$$
\sum_{j=1}^{k} \lambda_{j}^{s+\ln }(\Omega) \geq \frac{2}{n+2 s}(2 \pi)^{2 s}\left(\omega_{n}|\Omega|\right)^{-2 s / n} k^{1+2 s / n}\left(\ln \frac{(2 \pi)^{n} k}{\omega_{n}|\Omega|}-\frac{n}{n+2 s}\right).
$$

This proves the claim.
\end{proof}

\begin{remark}\label{ Remark 2.6}
   The lower bound in Theorem \ref{Theorem 2.5} has the correct leading-order growth. Indeed, by separating the $\ln k$ term, we obtain

$$
\begin{aligned}
\sum_{j=1}^{k} \lambda_{j}^{s+\ln }(\Omega) \geq & \frac{2}{n+2 s}(2 \pi)^{2 s}\left(\omega_{n}|\Omega|\right)^{-2 s / n} k^{1+2 s / n} \ln k \\
& +\frac{2}{n+2 s}(2 \pi)^{2 s}\left(\omega_{n}|\Omega|\right)^{-2 s / n}\left(\ln \frac{(2 \pi)^{n}}{\omega_{n}|\Omega|}-\frac{n}{n+2 s}\right) k^{1+2 s / n}.
\end{aligned}
$$

Hence,

$$
\sum_{j=1}^{k} \lambda_{j}^{s+\ln }(\Omega) \geq \frac{2}{n+2 s}(2 \pi)^{2 s}\left(\omega_{n}|\Omega|\right)^{-2 s / n} k^{1+2 s / n} \ln k+O\left(k^{1+2 s / n}\right) \quad \text { as } k \rightarrow \infty .
$$

In particular, the principal term of the above lower bound coincides with the expected asymptotic order

$$
k^{1+2 s / n} \ln k,
$$

and the leading constant is exactly

$$
\frac{2}{n+2 s}(2 \pi)^{2 s}\left(\omega_{n}|\Omega|\right)^{-2 s / n}.
$$
\end{remark}

\begin{corollary} \label{Corollary 2.7} (i) If

$$
k>\frac{\omega_{n}|\Omega|}{(2 \pi)^{n}} \exp \left(\frac{n}{n+2 s}\right),
$$

then

$$
\sum_{j=1}^{k} \lambda_{j}^{s+\ln }(\Omega)>0.
$$

In particular,

$$
\lambda_{k}^{s+\ln }(\Omega)>0 .
$$

(ii) If

$$
|\Omega|<\frac{(2 \pi)^{n}}{\omega_{n}} \exp \left(-\frac{n}{n+2 s}\right),
$$

then

$$
\lambda_{1}^{s+\ln }(\Omega)>0 .
$$
\end{corollary}

\begin{proof}
  (i) By Theorem \ref{Theorem 2.5},

$$
\sum_{j=1}^{k} \lambda_{j}^{s+\ln }(\Omega) \geq \frac{2}{n+2 s}(2 \pi)^{2 s}\left(\omega_{n}|\Omega|\right)^{-2 s / n} k^{1+2 s / n}\left(\ln \frac{(2 \pi)^{n} k}{\omega_{n}|\Omega|}-\frac{n}{n+2 s}\right).
$$

The prefactor is positive, so the right-hand side is strictly positive whenever

$$
\ln \frac{(2 \pi)^{n} k}{\omega_{n}|\Omega|}-\frac{n}{n+2 s}>0,
$$

that is,

$$
k>\frac{\omega_{n}|\Omega|}{(2 \pi)^{n}} \exp \left(\frac{n}{n+2 s}\right).
$$

Hence

$$
\sum_{j=1}^{k} \lambda_{j}^{s+\ln }(\Omega)>0.
$$

Since the eigenvalues are arranged in nondecreasing order, if

$$
\lambda_{k}^{s+\ln }(\Omega) \leq 0,
$$

then

$$
\lambda_{j}^{s+\ln }(\Omega) \leq 0 \quad \text { for all } 1 \leq j \leq k,
$$

and therefore

$$
\sum_{j=1}^{k} \lambda_{j}^{s+\ln }(\Omega) \leq 0,
$$

a contradiction. Thus

$$
\lambda_{k}^{s+\ln }(\Omega)>0.
$$

(ii) The assumption

$$
|\Omega|<\frac{(2 \pi)^{n}}{\omega_{n}} \exp \left(-\frac{n}{n+2 s}\right)
$$

is equivalent to

$$
1>\frac{\omega_{n}|\Omega|}{(2 \pi)^{n}} \exp \left(\frac{n}{n+2 s}\right).
$$

Applying part (i) with $k=1$, we obtain

$$
\lambda_{1}^{s+\ln }(\Omega)>0 .
$$
\end{proof}

\section{ Upper Bound Estimate}

Let $\eta_{0} \in C^{2}(\mathbb{R})$ be an increasing cutoff such that

$$
0 \leq \eta_{0} \leq 1, \quad \eta_{0}(t)=0 \text { for } t \leq 0, \quad \eta_{0}(t)=1 \text { for } t \geq 1,
$$

and $\left\|\eta_{0}\right\|_{C^{2}} \leq C$. For $\sigma>0$, define

$$
w_{\sigma}(x):=\eta_{0}\left(\frac{\rho(x)}{\sigma}\right), \quad x \in \mathbb{R}^{n},
$$

where $\rho(x):=\operatorname{dist}(x, \partial \Omega)$ for $x \in \Omega$ and $w_{\sigma}(x):=0$ for $x \notin \Omega$. Then

$$
0 \leq w_{\sigma} \leq 1, \quad w_{\sigma}(x)=0 \text { for } x \notin \Omega, \quad w_{\sigma}(x)=1 \text { whenever } \rho(x) \geq \sigma,
$$

so that $w_{\sigma}$ varies only in the boundary layer

$$
\{x \in \Omega: 0<\rho(x)<\sigma\}.
$$

Moreover,

$$
w_{\sigma} \in \mathcal{H}_{0}^{s+\ln }(\Omega), \quad w_{\sigma} \rightarrow 1 \quad \text { a.e. in } \quad \Omega \quad \text { as } \sigma \downarrow 0,
$$

and

\begin{equation}\label{3.1}
\int_{\Omega} w_{\sigma}^{2}dx
=
|\Omega|+O(|\Omega_{\sigma}|),
\end{equation}

where

\begin{equation}\label{3.2}
\Omega_t
:=
\{x\in\Omega:\operatorname{dist}(x,\partial\Omega)<t\}.
\end{equation}

For $z \in \mathbb{R}^{n}$, define

$$
v_{\sigma}(z, x):=w_{\sigma}(x) e^{-i x \cdot z}.
$$

The following cutoff plane-wave estimate is the key nonlocal ingredient in the proof of the upper bound. It may be viewed as the fractional-logarithmic counterpart of the corresponding estimates for the Laplacian and the fractional Laplacian.


\begin{lemma}\label{Lemma 3.1}
		Assume that there exist $t_0>0$ and $C_\Omega>0$ such that
		$$
		|\Omega_t|\leq C_\Omega t
		\qquad\text{for all }0<t\leq t_0.
		$$
		Then, for all sufficiently large $r$ and all sufficiently small
		$\sigma$, one has
		\begin{equation}\label{3.3}
			\int_{B_r}\int_\Omega
			\overline{v_\sigma(z,y)}
			(-\Delta)_y^{s+\ln}v_\sigma(z,y)
			\,dy\,dz
			=
			\mathcal{M}_{s,\ln}(r)
			\int_\Omega w_\sigma^2(y)\,dy
			+
			\Psi_{s,\ln}(r,\sigma),
		\end{equation}
		where
		\begin{equation}\label{3.4}
			\mathcal{M}_{s,\ln}(r)
			:=
			\int_{B_r}|z|^{2s}\ln|z|^2\,dz
			=
			n\omega_n r^{n+2s}
			\left(
			\frac{\ln r^2}{n+2s}
			-
			\frac{2}{(n+2s)^2}
			\right)>0.
		\end{equation}
		Moreover,
		\begin{equation}\label{3.5}
			|\Psi_{s,\ln}(r,\sigma)|
			\leq
			Cr^n
			\begin{cases}
				1, & 0<s<\frac12,\\
				(1+|\ln\sigma|)^2, & s=\frac12,\\
				\sigma^{1-2s}(1+|\ln\sigma|),
				& \frac12<s<1.
			\end{cases}
		\end{equation}
		In particular, if
		$$
		\sigma=c_0r^{-\beta},
		\qquad
		0<\beta<\min\{1,2s\},
		$$
		then, for all sufficiently large $r$,
		\begin{equation}\label{eq:psi-coupled}
			|\Psi_{s,\ln}(r,\sigma)|
			\leq
			C_1r^{n+2s}\sigma.
		\end{equation}
		Here $C>0$ depends only on $n$, $s$, and $\Omega$, while for fixed
		$c_0$ and $\beta$, the constant $C_1>0$ may also depend on
		$c_0$ and $\beta$.
\end{lemma}

\begin{proof}
We first rewrite the quadratic form associated with
$(-\Delta)^{s+\ln}$ for $v_\sigma$. We then identify the
principal term and estimate the corresponding remainder after
integration over $B_r$.
	
	 \medskip
\noindent\textbf{Step 1: Quadratic-form representation.}
For fixed $z\in B_r$, we set
$$
I_\sigma(z)
:=
\int_{\Omega}
\overline{v_\sigma(z,x)}
(-\Delta)^{s+\ln}v_\sigma(z,x)\,dx,
$$
where the above expression is understood in the quadratic-form
sense. By the quadratic-form representation associated with
$(-\Delta)^{s+\ln}$, we have
$$
I_\sigma(z)
=
\frac{1}{2}
\iint_{\mathbb{R}^{2n}}
|v_\sigma(z,x)-v_\sigma(z,y)|^2
K_{s+\ln}(|x-y|)\,dx\,dy,
$$
where
$$
K_{s+\ln}(r)
=
c_{n,s}
\bigl(b_{n,s}-2\ln r\bigr)r^{-n-2s},
\qquad
b_{n,s}
=
\frac{c_{n,s}'}{c_{n,s}}.
$$
	
\medskip
\noindent\textbf{Step 2: Separation of the principal term.}
Using $v_\sigma(z,x)=w_\sigma(x)e^{-ix\cdot z}$, we have
$$
\begin{aligned}
	|v_\sigma(z,x)-v_\sigma(z,y)|^2
	&=
	w_\sigma^2(x)+w_\sigma^2(y)
	-2w_\sigma(x)w_\sigma(y)
	\cos\bigl(z\cdot(x-y)\bigr) \\
	&=
	\bigl(w_\sigma^2(x)+w_\sigma^2(y)\bigr)
	\bigl(1-\cos\bigl(z\cdot(x-y)\bigr)\bigr) \\
	&\quad
	+\bigl(w_\sigma(x)-w_\sigma(y)\bigr)^2
	\cos\bigl(z\cdot(x-y)\bigr).
\end{aligned}
$$
Substituting this identity into the expression obtained in Step 1
and using the symmetry of the kernel, we obtain
$$
\begin{aligned}
	I_\sigma(z)
	&=
	\int_{\mathbb{R}^n} w_\sigma^2(x)
	\int_{\mathbb{R}^n}
	\bigl(1-\cos\bigl(z\cdot(x-y)\bigr)\bigr)
	K_{s+\ln}(|x-y|)\,dy\,dx \\
	&\quad
	+\frac{1}{2}
	\iint_{\mathbb{R}^{2n}}
	\bigl(w_\sigma(x)-w_\sigma(y)\bigr)^2
	\cos\bigl(z\cdot(x-y)\bigr)
	K_{s+\ln}(|x-y|)\,dx\,dy.
\end{aligned}
$$
To identify the first term, set $h=x-y$. Since
$K_{s+\ln}$ is radial, it is even, and the contribution involving
$\sin(z\cdot h)$ vanishes. By the kernel representation and the
Fourier multiplier characterization of $(-\Delta)^{s+\ln}$ in
\cite[Theorem 1.1(ii)--(iii)]{CCH26}, we obtain
$$
\int_{\mathbb{R}^n}
\bigl(1-\cos(z\cdot h)\bigr)
K_{s+\ln}(|h|)\,dh
=
|z|^{2s}\ln|z|^2.
$$
Therefore,
$$
I_\sigma(z)
=
|z|^{2s}\ln|z|^2
\int_\Omega w_\sigma^2(x)\,dx
+
R_{s,\ln}(z,\sigma),
$$
where
$$
R_{s,\ln}(z,\sigma)
:=
\frac{1}{2}
\iint_{\mathbb{R}^{2n}}
\bigl(w_\sigma(x)-w_\sigma(y)\bigr)^2
\cos\bigl(z\cdot(x-y)\bigr)
K_{s+\ln}(|x-y|)\,dx\,dy.
$$
	
\medskip
\noindent\textbf{Step 3: Integration with respect to $z$.}
Integrating the decomposition obtained in Step 2 over $B_r$, we get
$$
\int_{B_r} I_\sigma(z)\,dz
=
\left(
\int_{B_r}|z|^{2s}\ln|z|^2\,dz
\right)
\int_\Omega w_\sigma^2(x)\,dx
+
\int_{B_r}R_{s,\ln}(z,\sigma)\,dz.
$$
Using polar coordinates, we have
$$
\begin{aligned}
	\int_{B_r}|z|^{2s}\ln|z|^2\,dz
	&=
	n\omega_n\int_0^r
	\rho^{n+2s-1}\ln\rho^2\,d\rho \\
	&=
	n\omega_n r^{n+2s}
	\left(
	\frac{\ln r^2}{n+2s}
	-
	\frac{2}{(n+2s)^2}
	\right)
	=
	\mathcal{M}_{s,\ln}(r).
\end{aligned}
$$
In particular, $\mathcal{M}_{s,\ln}(r)>0$ for all sufficiently large
$r$. Renaming the integration variable $x$ as $y$ and defining
$$
\Psi_{s,\ln}(r,\sigma)
:=
\int_{B_r}R_{s,\ln}(z,\sigma)\,dz,
$$
we obtain
$$
\int_{B_r} I_\sigma(z)\,dz
=
\mathcal{M}_{s,\ln}(r)
\int_\Omega w_\sigma^2(y)\,dy
+
\Psi_{s,\ln}(r,\sigma).
$$
	
\medskip
\noindent\textbf{Step 4: Estimate of the remainder.}
By the definitions of $R_{s,\ln}(z,\sigma)$ and
$\Psi_{s,\ln}(r,\sigma)$, and since $|\cos\theta|\leq1$, we have
$$
|\Psi_{s,\ln}(r,\sigma)|
\leq
\frac12\int_{B_r}
\iint_{\mathbb{R}^{2n}}
\bigl(w_\sigma(x)-w_\sigma(y)\bigr)^2
\left|K_{s+\ln}(|x-y|)\right|
\,dx\,dy\,dz.
$$
Since
$$
|K_{s+\ln}(\rho)|
\leq
C\frac{1+|\ln\rho|}{\rho^{n+2s}},
\qquad \rho>0,
$$
and $|B_r|=\omega_n r^n$, it follows that
$$
|\Psi_{s,\ln}(r,\sigma)|
\leq
Cr^n E_\sigma,
$$
where
$$
E_\sigma
:=
\iint_{\mathbb{R}^{2n}}
\bigl(w_\sigma(x)-w_\sigma(y)\bigr)^2
\frac{1+|\ln|x-y||}{|x-y|^{n+2s}}
\,dx\,dy.
$$

For $h\in\mathbb{R}^n$, set
$$
A_\sigma(h)
:=
\int_{\mathbb{R}^n}
|w_\sigma(x+h)-w_\sigma(x)|^2\,dx.
$$
Since $\rho(x)=\operatorname{dist}(x,\partial\Omega)$ is
$1$-Lipschitz, $\eta_0(0)=0$, and $\|\eta_0'\|_\infty\leq C$, the
zero extension of $w_\sigma$ satisfies
$$
|w_\sigma(x+h)-w_\sigma(x)|
\leq
C\min\left\{1,\frac{|h|}{\sigma}\right\}
$$
for all $x,h\in\mathbb{R}^n$.

Let
$$
D_h
:=
\left\{
x\in\mathbb{R}^n:
w_\sigma(x+h)\neq w_\sigma(x)
\right\}.
$$
If $x\in D_h\cap\Omega$, then necessarily
$\operatorname{dist}(x,\partial\Omega)\leq\sigma+|h|$; otherwise
$w_\sigma(x)=w_\sigma(x+h)=1$. If $x\in D_h\setminus\Omega$,
then $x+h\in\Omega$ and
$\operatorname{dist}(x+h,\partial\Omega)\leq|h|$. Hence
$$
D_h
\subset
\Omega_{\sigma+|h|}
\cup
\bigl(\Omega_{|h|}-h\bigr).
$$
Set
$$
\rho_0:=\min\left\{\frac{t_0}{2},\frac12\right\}
$$
and assume that $0<\sigma\leq\rho_0$. For $|h|\leq\rho_0$, the
boundary-layer estimate yields
$$
|D_h|
\leq
C(\sigma+|h|).
$$
Combining the previous two estimates, and using
$\|w_\sigma\|_{L^2(\mathbb{R}^n)}^2\leq|\Omega|$ when
$|h|>\rho_0$, we obtain
$$
A_\sigma(h)
\leq
C
\begin{cases}
	|h|^2/\sigma, & |h|\leq\sigma,\\
	|h|, & \sigma<|h|\leq\rho_0,\\
	1, & |h|>\rho_0.
\end{cases}
$$

Using the change of variables $h=y-x$ and polar coordinates, we
therefore obtain
$$
E_\sigma
\leq
C\left[
\frac1{\sigma}
\int_0^\sigma
\rho^{1-2s}(1+|\ln\rho|)\,d\rho
+
\int_\sigma^{\rho_0}
\rho^{-2s}(1+|\ln\rho|)\,d\rho
+1
\right].
$$
Elementary integration then gives
$$
E_\sigma
\leq
C
\begin{cases}
	1, & 0<s<\frac12,\\
	(1+|\ln\sigma|)^2, & s=\frac12,\\
	\sigma^{1-2s}(1+|\ln\sigma|),
	& \frac12<s<1.
\end{cases}
$$
Consequently,
$$
|\Psi_{s,\ln}(r,\sigma)|
\leq
Cr^n
\begin{cases}
	1, & 0<s<\frac12,\\
	(1+|\ln\sigma|)^2, & s=\frac12,\\
	\sigma^{1-2s}(1+|\ln\sigma|),
	& \frac12<s<1.
\end{cases}
$$

Finally, let
$$
\sigma=c_0r^{-\beta},
\qquad
0<\beta<\min\{1,2s\}.
$$
For $0<s<\frac12$, the condition $\beta<2s$ gives
$r^n=O(r^{n+2s}\sigma)$. For $s=\frac12$, the condition
$\beta<1$ and the fact that logarithmic powers are dominated by
positive powers of $r$ give
$$
r^n(1+|\ln\sigma|)^2
=
O(r^{n+1}\sigma).
$$
If $\frac12<s<1$, since $\beta<1$,
$$
r^n\sigma^{1-2s}(1+|\ln\sigma|)
=
O(r^{n+2s}\sigma).
$$
Hence, in all cases,
$$
|\Psi_{s,\ln}(r,\sigma)|
\leq
C_1r^{n+2s}\sigma
$$
for all sufficiently large $r$. This completes the proof.
	
\end{proof}

\begin{theorem}\label{Theorem 3.2}
   Let $\Omega \subset \mathbb{R}^{n}$ be a bounded domain, and let

$$
\lambda_{1}^{s+\ln }(\Omega) \leq \lambda_{2}^{s+\ln }(\Omega) \leq \cdots \uparrow+\infty
$$

be the Dirichlet eigenvalues of $(-\Delta)^{s+\ln }$ on $\Omega$, where $0<s<1$. Assume that there exist $t_{0}>0$ and $C_{\Omega}>0$ such that

$$
\left|\Omega_{t}\right| \leq C_{\Omega} t \quad \text { for all } 0<t \leq t_{0},
$$

where $\Omega_{t}$ is defined in \eqref{3.2}. Then there exists a constant $C=C(n, s, \Omega)>0$ such that, for every integer $k> \frac{\omega_{n}|\Omega|}{(2 \pi)^{n}} \exp \left(\frac{n}{n+2 s}\right)$, we have

$$
\sum_{j=1}^{k} \lambda_{j}^{s+\ln }(\Omega) \leq \frac{2}{n+2 s}(2 \pi)^{2 s}\left(\omega_{n}|\Omega|\right)^{-2 s / n} k^{1+2 s / n} \ln k+C k^{1+2 s / n}.
$$
\end{theorem}

\begin{proof}
   Let $\left\{\phi_{j}\right\}_{j \geq 1}$ be a real-valued orthonormal basis of eigenfunctions in $L^{2}(\Omega)$ associated with $\left\{\lambda_{j}^{s+\ln }(\Omega)\right\}_{j \geq 1}$. We use the normalized Fourier transform

$$
\widehat{u}(\xi):=(2 \pi)^{-n / 2} \int_{\mathbb{R}^{n}} u(x) e^{-i x \cdot \xi} d x, \quad \xi \in \mathbb{R}^{n}
$$

Under this convention, the Fourier symbol of $(-\Delta)^{s+\ln }$ is

$$
m(\xi):=|\xi|^{2 s} \ln |\xi|^{2}.
$$

\medskip
\noindent
{\bf Step 1: Projection onto the orthogonal complement of the first $k$ eigenfunctions.} Set

$$
\Phi_{k}(x, y):=\sum_{j=1}^{k} \phi_{j}(x) \phi_{j}(y), \quad x, y \in \Omega,
$$

and denote by $\mathcal{F}_{x}\left(\Phi_{k}\right)(z, y)$ the Fourier transform of $\Phi_{k}$ with respect to the $x$-variable. Define

$$
v_{\sigma, k}(z, y):=v_{\sigma}(z, y)-(2 \pi)^{n / 2} \mathcal{F}_{x}\left(w_{\sigma} \Phi_{k}\right)(z, y).
$$

Then, for every fixed $z\in\mathbb{R}^{n}$, the function
$y\mapsto v_{\sigma,k}(z,y)$ belongs to
$\mathcal{H}_{0}^{s+\ln}(\Omega)$ and is orthogonal in
$L^{2}(\Omega)$ to $\phi_{1},\ldots,\phi_{k}$.
Since $k$ is in the range stated in the theorem,
$\lambda_{k+1}^{s+\ln}(\Omega)>0$. By the Rayleigh-Ritz principle,
$$
\int_{\Omega}
\overline{v_{\sigma,k}(z,y)}
(-\Delta)_{y}^{s+\ln}v_{\sigma,k}(z,y)\,dy
\geq
\lambda_{k+1}^{s+\ln}(\Omega)
\int_{\Omega}|v_{\sigma,k}(z,y)|^{2}\,dy.
$$
Integrating over
$B_r:=\{z\in\mathbb{R}^{n}:|z|<r\}$, we obtain
\begin{equation}\label{eq:rayleigh-upper}
	\lambda_{k+1}^{s+\ln}(\Omega)
	\int_{B_r}\int_{\Omega}|v_{\sigma,k}(z,y)|^{2}\,dy\,dz
	\leq
	\int_{B_r}\int_{\Omega}
	\overline{v_{\sigma,k}(z,y)}
	(-\Delta)_{y}^{s+\ln}v_{\sigma,k}(z,y)\,dy\,dz.
\end{equation}

\medskip
\noindent
{\bf Step 2: Estimate of the denominator.} Using the orthogonality of the spectral projection, we have

$$
\int_{B_{r}} \int_{\Omega}\left|v_{\sigma, k}(z, y)\right|^{2} d y d z=\int_{B_{r}} \int_{\Omega}\left|v_{\sigma}(z, y)\right|^{2} d y d z-(2 \pi)^{n} \sum_{j=1}^{k} \int_{B_{r}}\left|\widehat{\left(w_{\sigma} \phi_{j}\right)}(z)\right|^{2} d z.
$$

Since $\left|v_{\sigma}(z, y)\right|^{2}=w_{\sigma}(y)^{2}$, it follows that

$$
\int_{B_{r}} \int_{\Omega}\left|v_{\sigma}(z, y)\right|^{2} d y d z=\left|B_{r}\right| \int_{\Omega} w_{\sigma}^{2}(y) d y=\omega_{n} r^{n} \int_{\Omega} w_{\sigma}^{2}(y) d y.
$$

Hence

\begin{equation}\label{eq:denominator-upper}
	\begin{aligned}
		\int_{B_r}\int_\Omega
		|v_{\sigma,k}(z,y)|^2\,dy\,dz
		=
		\omega_n r^n\int_\Omega w_\sigma^2(y)\,dy 
		-(2\pi)^n
		\sum_{j=1}^k
		\int_{B_r}
		\left|
		\widehat{(w_\sigma\phi_j)}(z)
		\right|^2\,dz.
	\end{aligned}
\end{equation}

\medskip
\noindent
{\bf Step 3: Estimate of the numerator.} For each fixed $z \in B_{r}$, write

$$
P_{k} v_{\sigma}(z, \cdot)=\sum_{j=1}^{k} a_{j}(z) \phi_{j}, \quad a_{j}(z):=\left\langle v_{\sigma}(z, \cdot), \phi_{j}\right\rangle_{L^{2}(\Omega)}.
$$

Then

$$
v_{\sigma, k}(z, \cdot)=v_{\sigma}(z, \cdot)-P_{k} v_{\sigma}(z, \cdot),
$$

and, since $(-\Delta)^{s+\ln } \phi_{j}=\lambda_{j}^{s+\ln }(\Omega) \phi_{j}$, we have

$$
(-\Delta)^{s+\ln } P_{k} v_{\sigma}(z, \cdot)=\sum_{j=1}^{k} a_{j}(z) \lambda_{j}^{s+\ln }(\Omega) \phi_{j}.
$$

Using the self-adjointness of $(-\Delta)^{s+\ln}$, together with
the spectral expansion of $P_kv_\sigma(z,\cdot)$ in the
orthonormal eigenfunction basis $\{\phi_j\}_{j=1}^{k}$, we compute
$$
\begin{aligned}
	&\int_{\Omega}
	\overline{v_{\sigma,k}(z,y)}
	(-\Delta)_{y}^{s+\ln}v_{\sigma,k}(z,y)\,dy\\
	&=
	\int_{\Omega}
	\overline{v_{\sigma}(z,y)}
	(-\Delta)_{y}^{s+\ln}v_{\sigma}(z,y)\,dy
	-
	\sum_{j=1}^{k}
	\lambda_{j}^{s+\ln}(\Omega)|a_j(z)|^{2}.
\end{aligned}
$$

Finally, since
$$
a_j(z)
=
\langle v_\sigma(z,\cdot),\phi_j\rangle_{L^2(\Omega)}
=
\int_\Omega
w_\sigma(x)e^{-ix\cdot z}\phi_j(x)\,dx
=
(2\pi)^{n/2}
\widehat{(w_\sigma\phi_j)}(z),
$$
it follows that
$$
|a_j(z)|^{2}
=
(2\pi)^n
\left|\widehat{(w_\sigma\phi_j)}(z)\right|^{2}.
$$
Hence
$$
\begin{aligned}
	&\int_{\Omega}
	\overline{v_{\sigma,k}(z,y)}
	(-\Delta)_{y}^{s+\ln}v_{\sigma,k}(z,y)\,dy\\
	&=
	\int_{\Omega}
	\overline{v_{\sigma}(z,y)}
	(-\Delta)_{y}^{s+\ln}v_{\sigma}(z,y)\,dy\\
	&\quad
	-(2\pi)^n
	\sum_{j=1}^{k}
	\lambda_j^{s+\ln}(\Omega)
	\left|\widehat{(w_\sigma\phi_j)}(z)\right|^{2}.
\end{aligned}
$$
Integrating over $z\in B_r$ yields
\begin{align}\label{eq:numerator-upper}
	&\int_{B_r}\int_{\Omega}
	\overline{v_{\sigma,k}(z,y)}
	(-\Delta)_{y}^{s+\ln}v_{\sigma,k}(z,y)\,dy\,dz
	\nonumber\\
	&=
	\int_{B_r}\int_{\Omega}
	\overline{v_{\sigma}(z,y)}
	(-\Delta)_{y}^{s+\ln}v_{\sigma}(z,y)\,dy\,dz
	\nonumber\\
	&\quad
	-(2\pi)^n
	\sum_{j=1}^{k}
	\lambda_j^{s+\ln}(\Omega)
	\int_{B_r}
	\left|\widehat{(w_\sigma\phi_j)}(z)\right|^{2}\,dz.
\end{align}

Thus the key point is to estimate the first term on the right-hand side.

\medskip
\noindent
{\bf Step 4: Main symbol term and cutoff error.}
By Lemma \ref{Lemma 3.1}, we have
$$
\int_{B_r}\int_\Omega
\overline{v_\sigma(z,y)}
(-\Delta)_y^{s+\ln}v_\sigma(z,y)
\,dy\,dz
=
\mathcal{M}_{s,\ln}(r)
\int_\Omega w_\sigma^2(y)\,dy
+
\Psi_{s,\ln}(r,\sigma).
$$
Combining this identity with \eqref{eq:numerator-upper}, we obtain
\begin{align}\label{eq:numerator-final-upper}
	&\int_{B_r}\int_{\Omega}
	\overline{v_{\sigma,k}(z,y)}
	(-\Delta)_{y}^{s+\ln}v_{\sigma,k}(z,y)\,dy\,dz
	\nonumber\\
	&=
	\mathcal{M}_{s,\ln}(r)
	\int_{\Omega}w_\sigma^2(y)\,dy
	+
	\Psi_{s,\ln}(r,\sigma)
	\nonumber\\
	&\quad
	-(2\pi)^n
	\sum_{j=1}^{k}
	\lambda_j^{s+\ln}(\Omega)
	\int_{B_r}
	\left|\widehat{(w_\sigma\phi_j)}(z)\right|^{2}\,dz.
\end{align}

\medskip
\noindent
{\bf Step 5: Insert the estimates into Rayleigh-Ritz.}
Let us fix
$$
\beta:=\frac12\min\{1,2s\},
\qquad
\alpha:=1-\frac{\beta}{n}.
$$
Then
$$
0<\beta<\min\{1,2s\},
\qquad
0<\alpha<1,
\qquad
\beta=n(1-\alpha).
$$
Moreover, let us define
$$
A:=(2\pi)(\omega_n|\Omega|)^{-1/n},
\qquad
r_k:=A(k+k^\alpha)^{1/n},
$$
and set
$$
\sigma:=c_0r_k^{-\beta},
$$

where $c_{0}>0$ will be chosen below. 
Since $0 \leq w_{\sigma} \leq 1$ and $w_{\sigma} \equiv 1$ on $\{\rho \geq \sigma\}$, we have

$$
0 \leq|\Omega|-\int_{\Omega} w_{\sigma}^{2} d y \leq\left|\Omega_{\sigma}\right|, \quad \Omega_{\sigma}:=\{x \in \Omega: \operatorname{dist}(x, \partial \Omega)<\sigma\}.
$$

By assumption, there exists $t_{0}>0$ such that

$$
\left|\Omega_{t}\right| \leq C_{\Omega} t \quad \text { for all } 0<t \leq t_{0}.
$$

If

$$
k \geq K_{0}:=\left(\frac{c_{0} A^{-\beta}}{t_{0}}\right)^{n / \beta},
$$

then $\sigma \leq t_{0}$, and hence

$$
|\Omega|-\int_{\Omega} w_{\sigma}^{2} d y \leq C_{\Omega} \sigma=C_{\Omega} c_{0} r_{k}^{-\beta}.
$$

Therefore,

$$
\begin{aligned}
\omega_{n} r_{k}^{n} \int_{\Omega} w_{\sigma}^{2} d y-(2 \pi)^{n} k & =(2 \pi)^{n} k^{\alpha}-\omega_{n} r_{k}^{n}\left(|\Omega|-\int_{\Omega} w_{\sigma}^{2} d y\right) \\
& \geq(2 \pi)^{n} k^{\alpha}-\omega_{n} C_{\Omega} c_{0} r_{k}^{n-\beta}.
\end{aligned}
$$

Since $\beta=n(1-\alpha)$, we have
$$
1-\frac{\beta}{n}=\alpha.
$$
Moreover, $\beta<1\leq n$, so $n-\beta>0$. Hence, using
$k+k^\alpha\leq2k$ for $k\geq1$, we obtain
$$
\begin{aligned}
	r_k^{n-\beta}
	&=
	A^{n-\beta}(k+k^\alpha)^{1-\beta/n}\\
	&=
	A^{n-\beta}(k+k^\alpha)^\alpha\\
	&\leq
	A^{n-\beta}2^\alpha k^\alpha.
\end{aligned}
$$
Therefore,
$$
\omega_n r_k^n\int_\Omega w_\sigma^2\,dy-(2\pi)^nk
\geq
\left(
(2\pi)^n
-\omega_nC_\Omega c_0A^{n-\beta}2^\alpha
\right)k^\alpha.
$$

Consequently, if

$$
0<c_0<
\frac{(2\pi)^n}
{\omega_nC_\Omega A^{n-\beta}2^\alpha}.$$

then

$$
\omega_{n} r_{k}^{n} \int_{\Omega} w_{\sigma}^{2} d y-(2 \pi)^{n} k>0 \quad \text { for all } k \geq K_{0}.
$$

By Plancherel's formula,

$$
\int_{B_{r_k}}\left|\widehat{\left(w_{\sigma} \phi_{j}\right)}(z)\right|^{2} d z \leq \int_{\mathbb{R}^{n}}\left|\widehat{\left(w_{\sigma} \phi_{j}\right)}(z)\right|^{2} d z=\int_{\Omega}\left|w_{\sigma} \phi_{j}\right|^{2} d x \leq 1,
$$

then

\begin{equation}\label{eq:positive-denominator-upper}
\omega_{n} r_k^{n} \int_{\Omega} w_{\sigma}^{2} d y-(2 \pi)^{n} \sum_{j=1}^{k} \int_{B_{r_k}}\left|\widehat{\left(w_{\sigma} \phi_{j}\right)}(z)\right|^{2} d z \geq \omega_{n} r_k^{n} \int_{\Omega} w_{\sigma}^{2} d y-(2 \pi)^{n} k>0, \quad k \geq K_{0}. 
\end{equation}

Set
$$
A_1
:=
\mathcal{M}_{s,\ln}(r_k)
\int_{\Omega}w_\sigma^2\,dy
+
\Psi_{s,\ln}(r_k,\sigma),
$$
and
$$
A_2
:=
\omega_n r_k^n
\int_{\Omega}w_\sigma^2\,dy.
$$
For $1\leq j\leq k$, let
$$
I_j
:=
\int_{B_{r_k}}
\left|\widehat{(w_\sigma\phi_j)}(z)\right|^2\,dz.
$$
Applying \eqref{eq:rayleigh-upper},
\eqref{eq:denominator-upper}, and
\eqref{eq:numerator-final-upper}
with $r=r_k$, we obtain
$$
\begin{aligned}
	0
	&\leq
	A_1-(2\pi)^n\sum_{j=1}^{k}
	\lambda_j^{s+\ln}(\Omega)I_j
	-
	\lambda_{k+1}^{s+\ln}(\Omega)
	\left(
	A_2-(2\pi)^n\sum_{j=1}^{k}I_j
	\right)\\
	&=
	A_1
	-
	A_2\lambda_{k+1}^{s+\ln}(\Omega)
	+
	(2\pi)^n
	\sum_{j=1}^{k}
	\left(
	\lambda_{k+1}^{s+\ln}(\Omega)
	-
	\lambda_j^{s+\ln}(\Omega)
	\right)I_j.
\end{aligned}
$$
Since $I_j\leq1$ and
$\lambda_{k+1}^{s+\ln}(\Omega)
-\lambda_j^{s+\ln}(\Omega)\geq0$, we obtain
$$
\begin{aligned}
	0
	&\leq
	A_1
	-
	A_2\lambda_{k+1}^{s+\ln}(\Omega)
	+
	(2\pi)^n
	\sum_{j=1}^{k}
	\left(
	\lambda_{k+1}^{s+\ln}(\Omega)
	-
	\lambda_j^{s+\ln}(\Omega)
	\right)\\
	&=
	A_1
	-
	\lambda_{k+1}^{s+\ln}(\Omega)
	\left(A_2-(2\pi)^nk\right)
	-
	(2\pi)^n
	\sum_{j=1}^{k}\lambda_j^{s+\ln}(\Omega).
\end{aligned}
$$
Since
$\lambda_{k+1}^{s+\ln}(\Omega)>0$ and
$A_2-(2\pi)^nk>0$, it follows that
\begin{equation}\label{eq:preliminary-upper}
	\sum_{j=1}^{k}\lambda_j^{s+\ln}(\Omega)
	\leq
	(2\pi)^{-n}
	\mathcal{M}_{s,\ln}(r_k)
	\int_{\Omega}w_\sigma^2\,dy
	+
	(2\pi)^{-n}
	\Psi_{s,\ln}(r_k,\sigma).
\end{equation}

For each $k \geq K_{0}$, define

$$
c_{1, k}:=\omega_{n} r_{k}^{n} \int_{\Omega} w_{\sigma}^{2} d y-(2 \pi)^{n} k>0.
$$

Then

$$
\omega_{n} r_{k}^{n} \int_{\Omega} w_{\sigma}^{2} d y=(2 \pi)^{n} k+c_{1, k}.
$$

Moreover, since $0 \leq w_{\sigma} \leq 1$, we have

$$
0<c_{1, k} \leq \omega_{n} r_{k}^{n}|\Omega|-(2 \pi)^{n} k=(2 \pi)^{n} k^{\alpha}.
$$

Hence, using \eqref{3.4},

$$
\begin{aligned}
(2 \pi)^{-n} \mathcal{M}_{s, \ln }\left(r_{k}\right) \int_{\Omega} w_{\sigma}^{2} d y & =(2 \pi)^{-n} n \omega_{n} r_{k}^{n+2 s}\left(\frac{\ln r_{k}^{2}}{n+2 s}-\frac{2}{(n+2 s)^{2}}\right) \int_{\Omega} w_{\sigma}^{2} d y \\
& =n\left(k+(2 \pi)^{-n} c_{1, k}\right) r_{k}^{2 s}\left(\frac{\ln r_{k}^{2}}{n+2 s}-\frac{2}{(n+2 s)^{2}}\right).
\end{aligned}
$$

Since

$$
r_{k}=(2 \pi)\left(\omega_{n}|\Omega|\right)^{-1 / n}\left(k+k^{\alpha}\right)^{1 / n},
$$

we have

$$
r_{k}^{2 s}=(2 \pi)^{2 s}\left(\omega_{n}|\Omega|\right)^{-2 s / n}\left(k+k^{\alpha}\right)^{2 s / n}=(2 \pi)^{2 s}\left(\omega_{n}|\Omega|\right)^{-2 s / n} k^{2 s / n}\left(1+\frac{2 s}{n} k^{\alpha-1}+O\left(k^{2 \alpha-2}\right)\right),
$$

and

$$
\begin{aligned}
\ln r_{k}^{2} & =\ln \left((2 \pi)^{2}\left(\omega_{n}|\Omega|\right)^{-2 / n}\left(k+k^{\alpha}\right)^{2 / n}\right) \\
& =\frac{2}{n} \ln k+\ln \left((2 \pi)^{2}\left(\omega_{n}|\Omega|\right)^{-2 / n}\right)+\frac{2}{n} k^{\alpha-1}+O\left(k^{2 \alpha-2}\right).
\end{aligned}
$$

Furthermore, since $0<c_{1, k} \leq(2 \pi)^{n} k^{\alpha}$ and $0<\alpha<1$, it follows that

$$
k+(2 \pi)^{-n} c_{1, k}=k\left(1+O\left(k^{\alpha-1}\right)\right).
$$

Therefore,

$$
\begin{aligned}
(2 \pi)^{-n} \mathcal{M}_{s, \ln }\left(r_{k}\right) \int_{\Omega} w_{\sigma}^{2} d y= & \frac{2}{n+2 s}(2 \pi)^{2 s}\left(\omega_{n}|\Omega|\right)^{-2 s / n} k^{1+2 s / n} \ln k \\
& +O\left(k^{1+2 s / n}\right)+O\left(k^{\alpha+2 s / n} \ln k\right).
\end{aligned}
$$

Since $0<\alpha<1$, we have
$$
k^{\alpha+2s/n}\ln k
=
O\left(k^{1+2s/n}\right).
$$
Therefore, for all sufficiently large $k$,

\begin{equation}\label{eq:main-term-upper}
(2 \pi)^{-n} \mathcal{M}_{s, \ln }(r_k) \int_{\Omega} w_{\sigma}^{2} d y \leq \frac{2}{n+2 s}(2 \pi)^{2 s}\left(\omega_{n}|\Omega|\right)^{-2 s / n} k^{1+2 s / n} \ln k+C k^{1+2 s / n}.
\end{equation}

On the other hand, by \eqref{eq:psi-coupled},
$$
|\Psi_{s,\ln}(r_k,\sigma)|
\leq
C_1r_k^{n+2s}\sigma
=
C_1c_0r_k^{n+2s-\beta}.
$$
Since $r_k=O(k^{1/n})$, we obtain
\begin{equation}\label{eq:remainder-upper}
	|\Psi_{s,\ln}(r_k,\sigma)|
	\leq
	Ck^{1+(2s-\beta)/n}
	\leq
	Ck^{1+2s/n}.
\end{equation}

Combining \eqref{eq:preliminary-upper},
\eqref{eq:main-term-upper}, and
\eqref{eq:remainder-upper}, we obtain, for all sufficiently large
$k$,
$$
\sum_{j=1}^{k}\lambda_j^{s+\ln}(\Omega)
\leq
\frac{2}{n+2s}(2\pi)^{2s}
(\omega_n|\Omega|)^{-2s/n}
k^{1+2s/n}\ln k
+
Ck^{1+2s/n}.
$$
There remain only finitely many integers in the range stated in
the theorem. Enlarging $C$ if necessary therefore gives the same
estimate for all such $k$. This completes the proof.
\end{proof}

\section{Numerical illustration in one dimension}
\label{sec:numerical-illustration}

In this section, we complement the eigenvalue sum estimates established
in Theorems~\ref{Theorem 2.5} and~\ref{Theorem 3.2} with a
one-dimensional numerical study. We construct a finite element
approximation of the corresponding Dirichlet eigenvalue problem and
examine whether the computed eigenvalue sums exhibit the growth
predicted by the theoretical results.

\subsection{Continuous spectral setting}

Let $\Omega=(0,1)$ and let $0<s<1$. As recalled in the introduction,
the fractional-logarithmic Laplacian is defined in~\cite{CCH26} as the
derivative of the fractional Laplacian with respect to its order
parameter. The Dirichlet spectral problem associated with
$(-\Delta)^{s+\ln}$ is well posed in the corresponding energy
framework. Its eigenvalues
\[
\lambda_1^{s+\ln}(\Omega)
\leq
\lambda_2^{s+\ln}(\Omega)
\leq
\cdots
\uparrow +\infty
\]
satisfy, in dimension one, the Weyl-type asymptotic law
\[
\lambda_k^{s+\ln}(\Omega)
\sim
2\pi^{2s}k^{2s}\ln k,
\qquad k\to\infty.
\]
By Proposition~\ref{Proposition 2.3}, we have
\begin{equation}
	\label{eq:sum-leading-term-num}
	\sum_{j=1}^k\lambda_j^{s+\ln}(\Omega)
	\sim
	\frac{2}{1+2s}\pi^{2s}
	k^{1+2s}\ln k,
	\qquad k\to\infty.
\end{equation}
Therefore, we write the principal term as
\begin{equation}
	\label{eq:leading-term-1d}
	A_s k^{1+2s}\ln k,
	\qquad
	A_s:=\frac{2}{1+2s}\pi^{2s}.
\end{equation}
Moreover, the explicit lower bound in
Theorem~\ref{Theorem 2.5} reduces to
\begin{equation}
	\label{eq:lower-bound-1d}
	L_s(k)
	:=
	A_s k^{1+2s}
	\left(
	\ln(\pi k)-\frac{1}{1+2s}
	\right),
	\qquad k\geq 1.
\end{equation}
In the present setting, $(2\pi)^{-1}\omega_1|\Omega|=1/\pi<1$.
Hence, the condition in Theorem~\ref{Theorem 2.5} is automatically
satisfied for every integer $k\geq1$.

\subsection{Finite element discretization of the fractional operator}

In this subsection, we construct a finite element approximation of the
integral fractional Laplacian on $\Omega=(0,1)$. We use conforming
$P_1$ finite elements on a uniform mesh. We refer, for instance,
to~\cite{BH17,BBNOS18,Dao26} for related finite element approaches to
one-dimensional fractional diffusion problems.

Let us consider the uniform partition
\[
0=x_0<x_1<\cdots<x_{N+1}=1,
\qquad
x_i=ih,
\qquad
h=\frac{1}{N+1}.
\]
Let $V_h$ be the conforming $P_1$ finite element space
\[
V_h :=
\left\{
v_h\in C(\mathbb{R}) :
v_h=0 \text{ in } \mathbb{R}\setminus(0,1),
\quad
v_h|_{[x_{i-1},x_i]}\in\mathbb{P}_1
\ \text{for } i=1,\ldots,N+1
\right\}.
\]
We denote its nodal basis by $\{\phi_i\}_{i=1}^{N}$.

For the integral fractional Laplacian of order $s$, let
$\mathbb A_h(s)$ denote the stiffness matrix associated with the
bilinear form
\begin{equation}
	\label{eq:frac-discrete-bilinear}
	a_h^s(u_h,v_h)
	=
	\frac{c_{1,s}}{2}
	\iint_{\mathbb R^2}
	\frac{
		(u_h(x)-u_h(y))(v_h(x)-v_h(y))
	}{
		|x-y|^{1+2s}
	}
	\,dy\,dx,
\end{equation}
where 
\[
c_{1,s}
=
\frac{
	2^{2s}s\,
	\Gamma\!\left(s+\frac12\right)
}{
	\sqrt{\pi}\,\Gamma(1-s)
}.
\]
Therefore, the entries of $\mathbb A_h(s)$ are given by
\begin{equation}
	\label{eq:frac-matrix-entry}
	(\mathbb A_h(s))_{ij}
	=
	\frac{c_{1,s}}{2}
	\iint_{\mathbb R^2}
	\frac{
		(\phi_i(x)-\phi_i(y))
		(\phi_j(x)-\phi_j(y))
	}{
		|x-y|^{1+2s}
	}
	\,dy\,dx,
	\qquad 1\leq i,j\leq N.
\end{equation}

In addition, the mass matrix is the standard finite element matrix
\begin{equation}
	\label{eq:mass-matrix-num}
	(\mathbb M_h)_{ij}
	=
	\int_0^1\phi_i(x)\phi_j(x)\,dx,
	\qquad 1\leq i,j\leq N.
\end{equation}

\subsection{Approximation of the fractional-logarithmic operator}

Since
\[
(-\Delta)^{s+\ln}
=
\frac{\partial}{\partial s}(-\Delta)^s,
\]
we approximate the fractional-logarithmic stiffness matrix by the
centered difference
\begin{equation}
	\label{eq:Ah-flog}
	\mathbb A_h^{s+\ln}
	:=
	\frac{
		\mathbb A_h(s+\varepsilon)
		-
		\mathbb A_h(s-\varepsilon)
	}{
		2\varepsilon
	},
\end{equation}
where $\varepsilon>0$ is chosen so that
$s\pm\varepsilon\in(0,1)$. This centered-difference approximation accounts for the full
dependence of the fractional stiffness matrix on $s$, including both
the power of the kernel and the normalization constant $c_{1,s}$. For fixed $h$, this centered difference is a second-order
approximation with respect to $\varepsilon$, provided that
$A_h(s)$ is sufficiently smooth as a function of $s$.

Then, we solve the generalized matrix eigenvalue problem
\begin{equation}
	\label{eq:discrete-eig-problem}
	\mathbb A_h^{s+\ln}U
	=
	\lambda_h\mathbb M_hU,
	\qquad
	U\in\mathbb R^N\setminus\{0\}.
\end{equation}
The computed eigenvalues are ordered increasingly:
\[
\lambda_{1,h}^{s+\ln}
\leq
\lambda_{2,h}^{s+\ln}
\leq
\cdots.
\]
Finally, we define the corresponding partial sums by
\begin{equation}
	\label{eq:partial-sums-num}
	S_k^h
	:=
	\sum_{j=1}^k\lambda_{j,h}^{s+\ln}.
\end{equation}

\subsection{Numerical illustration of the eigenvalue sum estimates}

Now, we present some numerical experiments on $\Omega=(0,1)$. The first
objective is to compare the computed partial sums with the explicit
lower bound~\eqref{eq:lower-bound-1d}. The second is to examine whether
the difference between the numerical sums and the principal
term~\eqref{eq:leading-term-1d} is compatible with the correction order
appearing in Theorems~\ref{Theorem 2.5} and~\ref{Theorem 3.2}. These experiments are meant to illustrate numerically the theoretical
estimates and the asymptotic behavior predicted by the analysis.

\paragraph{Variation with respect to the fractional parameter.}

First, let us fix a uniform mesh with $h=2^{-9}$ and consider
\[
s\in
\left\{
\frac14,\frac12,\frac34
\right\}.
\]
The centered difference~\eqref{eq:Ah-flog} is computed with
\[
\varepsilon
=
\min
\left\{
10^{-3},
\frac{s}{4},
\frac{1-s}{4}
\right\}.
\]

This choice keeps the finite-difference step small while ensuring that
both $s-\varepsilon$ and $s+\varepsilon$ remain safely inside the
interval $(0,1)$.

Furthermore, we perform a sensitivity test for $s=\frac12$ and $h=2^{-9}$,
using
\[
\varepsilon\in
\left\{2\times10^{-3},10^{-3},5\times10^{-4}\right\}.
\]
For $k=20,50,80$, the relative differences between the partial sums
computed with $\varepsilon=10^{-3}$ and
$\varepsilon=5\times10^{-4}$ are of order $10^{-5}$. This indicates
that the choice $\varepsilon=10^{-3}$ yields stable values for the
partial sums over the range considered below.

We first compare, in the left column of
Figure~\ref{fig:multi-s-bounds}, the numerical sums $S_k^h$ with the
lower bound $L_s(k)$ defined in~\eqref{eq:lower-bound-1d} and with the
principal term $A_s k^{1+2s}\ln k$. To further examine the deviation
from the leading-order behavior, we plot in the right column the
normalized absolute remainder
\[
\frac{
	\left|
	S_k^h-A_s k^{1+2s}\ln k
	\right|
}{
	k^{1+2s}
}.
\]
Across the three values of $s$, the normalized remainder remains
bounded and decreases over the displayed range of indices. This
behavior is consistent with a correction of order $k^{1+2s}$, as
suggested by the lower and upper estimates in
Theorems~\ref{Theorem 2.5} and~\ref{Theorem 3.2}.

\begin{figure}[htbp]
	\centering
	\includegraphics[width=\textwidth]
	{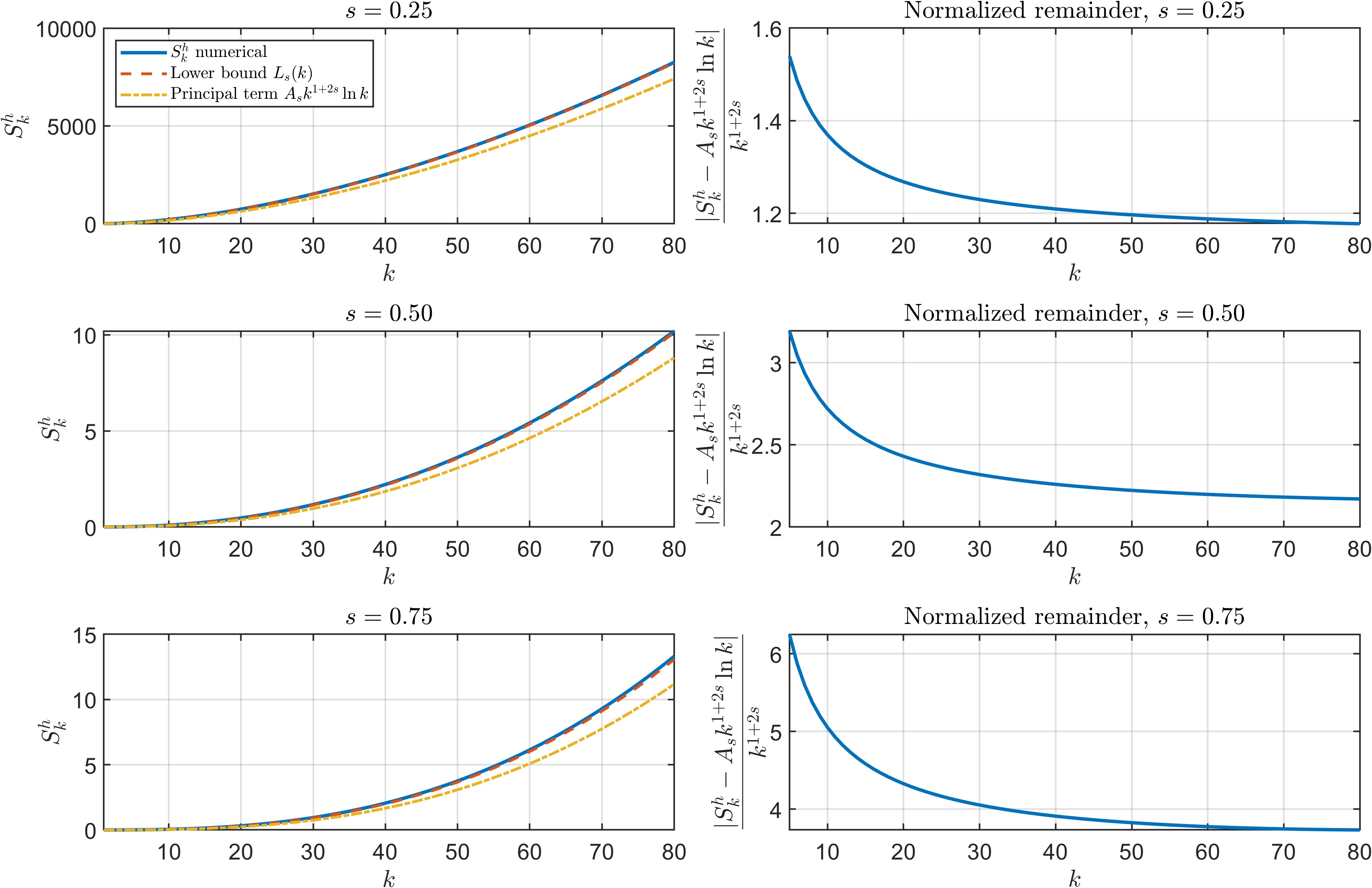}
	\caption{
		Numerical spectral sums for
		$s\in\{\frac14,\frac12,\frac34\}$, with $h=2^{-9}$ and
		$\varepsilon=10^{-3}$. Left column: comparison of the
		numerical sums $S_k^h$ with the lower bound $L_s(k)$ and
		the principal term $A_s k^{1+2s}\ln k$, displayed for
		$1\leq k\leq80$. Right column: normalized absolute
		remainder
		$\left|S_k^h-A_s k^{1+2s}\ln k\right|/k^{1+2s}$,
		displayed for $5\leq k\leq80$.
	}
	\label{fig:multi-s-bounds}
\end{figure}

Moreover, the computations show a clear change in the spectral growth as
$s$ varies. Over the displayed range, the explicit lower bound remains
below and close to the computed sums, while the principal term captures
their dominant growth.
\paragraph{Sensitivity with respect to the mesh size.}

We next fix $s=\frac12$ and compare the meshes
\[
h\in
\left\{
2^{-8},
2^{-9},
2^{-10}
\right\},
\]
using $\varepsilon=10^{-3}$. In Figure~\ref{fig:mesh-sensitivity}, we
display, in the left panel, the ratio
\[
\frac{
	S_k^h
}{
	A_s k^{1+2s}\ln k
},
\]
and, in the right panel, the absolute remainder
\[
\left|
S_k^h-A_s k^{1+2s}\ln k
\right|.
\]
The ratio is displayed for $2\leq k\leq80$, since its denominator
vanishes at $k=1$, whereas the absolute remainder is displayed for
$1\leq k\leq80$.

\begin{figure}[htbp]
	\centering
	\includegraphics[width=0.95\textwidth]
	{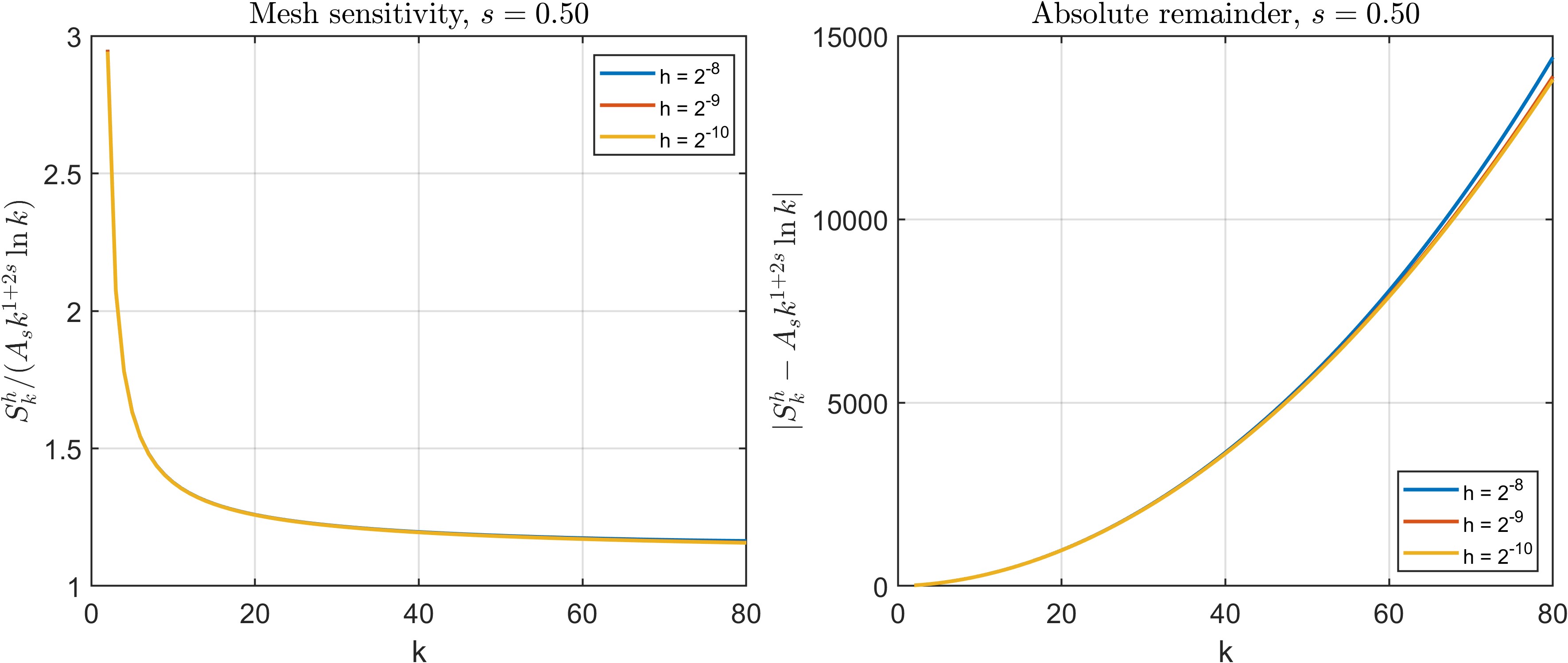}
	\caption{
		Sensitivity with respect to the mesh size for
		$s=\frac12$ and $\varepsilon=10^{-3}$. Left: ratio
		$S_k^h/(A_s k^{1+2s}\ln k)$ for $2\leq k\leq80$.
		Right: absolute remainder
		$\left|S_k^h-A_s k^{1+2s}\ln k\right|$
		for $1\leq k\leq80$. The three meshes are
		$h=2^{-8}$, $2^{-9}$, and $2^{-10}$.
	}
	\label{fig:mesh-sensitivity}
\end{figure}

We observe very good agreement among the results obtained with the
three meshes over the displayed range of indices. In particular, the
curves corresponding to $h=2^{-9}$ and $h=2^{-10}$ are almost
indistinguishable. This agreement indicates that the numerical results
are stable under mesh refinement for low and intermediate eigenvalue
indices.

\paragraph{Numerical conclusions.}

These experiments indicate that
\begin{itemize}
	\item the explicit lower bound in
	Theorem~\ref{Theorem 2.5} remains below and close to the computed
	partial sums over the considered range;
	
	\item the principal term $A_s k^{1+2s}\ln k$ captures the dominant
	growth of the numerical eigenvalue sums;
	
	\item the normalized remainder is compatible with a correction of
	order $k^{1+2s}$, as suggested by
	Theorems~\ref{Theorem 2.5} and~\ref{Theorem 3.2};
	
	\item the computed partial sums remain stable with respect to the
	choice of $\varepsilon$ over the tested range;
	
	\item the numerical results remain stable under mesh refinement for
	low and intermediate eigenvalue indices.
\end{itemize}


\noindent{\bf Conflicts of interest:} The authors declare that they have no conflicts of interest regarding this work.\\
{\bf Data availability:} This paper has no associated data.

\end{document}